\documentclass[12pt]{article}

\usepackage{amssymb}
\usepackage{amsmath}
\usepackage{amsthm}
\usepackage{amsfonts}
\usepackage{bbm}

\newcommand{\eee}{{\rm e}}

\DeclareMathOperator{\1}{\mathbbm{1}}

\newcommand{\mmp}{\mathbb{P}}

\newcommand{\dod}{\overset{{\rm d}}{\to}}

\newcommand{\me}{\mathbb{E}}
\newcommand{\mr}{\mathbb{R}}
\newcommand{\mn}{\mathbb{N}}

\newtheorem{thm}{Theorem}[section]
\newtheorem{lemma}[thm]{Lemma}

\newtheorem{assertion}[thm]{Proposition}

\theoremstyle{definition}

\theoremstyle{remark}

\begin{document}
\title{Functional limit theorems for random Lebesgue-Stieltjes convolutions}
\date{\today}

\author{Alexander Iksanov\footnote{ Faculty of Computer Science and Cybernetics, Taras Shevchenko National University of Kyiv, 01033 Kyiv, Ukraine, e-mail:
iksan@univ.kiev.ua} ~~~and~~ Wissem Jedidi \footnote{Department of Statistics and Operations Research, College of Science, King Saud University, P.O. Box 2455, Riyadh 11451, Saudi Arabia, e-mail: wjedidi@ksu.edu.sa} }

\maketitle

\begin{abstract}
\noindent
We prove joint functional limit theorems in the Skorokhod space equipped with the $J_1$-topology for successive Lebesgue-Stieltjes convolutions of nondecreasing stochastic processes with themselves. These convolutions arise naturally in coupled branching random walks, where the displacements of individuals relative to their mother's position are given by the underlying point process rather than its copy. Surprisingly, the numbers of individuals in the $j$th generation, with positions less than or equal to $t$, exhibit remarkably similar distributional behavior in both standard branching random walks and coupled branching random walks as $t$ tends to infinity.
\end{abstract}


\section{Introduction}\label{sect:intro}

Let $X:=(X(t))_{t\geq 0}$ be an almost surely (a.s.) nondecreasing and right-continuous stochastic process. For instance, $X$ may be a subordinator (an a.s.\ nondecreasing L\'{e}vy process), possibly a compound Poisson process with nonnegative jumps, or a locally finite counting process defined by $X(t):=\sum_{k\geq 1}\1_{\{T_k\leq t\}}$ for $t\geq 0$, where $(T_k)_{k\geq 1}$ is a sequence of nonnegative random variables, which is not necessarily monotone. One particular example appears in the situation where $(T_k)_{k\geq 1}$ is a globally perturbed random walk defined by
\begin{equation}\label{eq:pert}
T_1:=\eta_1\quad\text{and}\quad T_k:=\xi_1+\ldots+\xi_{k-1}+\eta_k,\quad k\geq 2.
\end{equation}
Here, $(\xi_k, \eta_k)_{k\geq 1}$ are independent copies of an $[0,\infty)\times [0,\infty)$-valued random vector $(\xi,\eta)$.

Put $X_1(t):=X(t)$ and, for $j\geq 2$, $$X_j(t):=\int_{[0,\,t]}X(t-y){\rm d}X_{j-1}(y)=\int_{[0,\,t]}X_{j-1}(t-y){\rm d}X(y),\quad t\geq 0.$$ The process $X_j$ is known as the {\it $j$-fold Lebesgue-Stieltjes convolution} of $X$ with itself, and is commonly denoted by $X^{\ast (j)}$. The article discusses the distributional convergence of $(X_1,\ldots, X_k)$ for any fixed $k$. 

A few words are in order about the motivation behind this work. First, the Lebesgue-Stieltjes convolution is a widely used operation in Probability and Analysis. We are unaware of any results concerning the distributional convergence of random convolutions, which is somewhat surprising. Secondly, and more importantly, such convolutions naturally arise in the definition of {\it coupled branching random walks}. Now we provide details and begin  by recalling what the standard (decoupled) branching random walks generated by the sequence $(T_k)_{k\geq 1}$ are. A population starts at time $0$ with a single individual, referred to as the ancestor. This ancestor produces offspring, which comprise the first generation of the population. The birth times of these offspring are given by the points $T_1$, $T_2, \ldots$ The first generation then gives rise to the second generation. The differences in birth times between individuals and their respective mothers are distributed according to copies of $(T_n)_{n \geq 1}$. Furthermore, these copies are independent for different mothers; this assumption, known as the {\it branching property}, validates the term `decoupled'. The second generation subsequently produces the third generation, and this pattern continues with each individual in these successive generations operating independently of others.

For $t\geq 0$ and $j\geq 1$, let $Y_j(t)$ denote the number of the $j$th generation individuals with birth times $\leq t$. Then $Y_1(t)=\sum_{k\geq 1}\1_{\{T_k\leq t\}}$ and, for $j\geq 2$,
\begin{equation}\label{eq:branch}
Y_j(t)=\sum_{k\geq 1}Y_{j-1}^{(k)}(t-T_k)\1_{\{T_k\leq t\}},\quad t\geq 0,
\end{equation}
where $Y_{j-1}^{(k)}(t)$ is the number of the $j$th generation individuals who are descendants of the first-generation individual with birth time $T_k$, and whose birth times fall in $[T_k, T_k +t]$. By the branching property, $(Y^{(1)}_{j-1}(t))_{t\geq 0}$, $(Y^{(2)}_{j-1}(t))_{t\geq 0},\ldots$ are independent copies of $(Y_{j-1}(t))_{t\geq 0}$, which are also independent of $(T_k)_{k\geq 1}$. The latter ensures that if $t\mapsto \me [Y_1(t)]$ is finite for each $t\geq 0$, then, for $j\geq 2$, the function $t\mapsto \me [Y_j(t)]$ is the $j$-fold Lebesgue-Stieltjes convolution of $t\mapsto \me [Y_1(t)]$ with itself.

To define a coupled branching random walk, we use the same sequence $(T_k)_{k\geq 1}$ for all individuals rather than its independent copies. Then, for $t\geq 0$ and $j\geq 1$, the variable $X_j(t)$ is exactly the number of the $j$th generation individuals with birth times $\leq t$. Of course, the function $t\mapsto \me [X_j(t)]$ provided it is finite for $t\geq 0$ is no longer the $j$-fold Lebesgue-Stieltjes convolution of $t\mapsto \me [X_1(t)]$. It can be checked that if $(T_k)_{k\geq 1}$ is a standard random walk, that is, the globally perturbed random walk with $\xi=\eta$, then $\me [X_2(t)]=\me [X(t/2)]+2\int_{[0,\,t]}\me [X(t-y)]{\rm d}\me [X(y/2)]$ for $t\geq 0$.

\section{Main results}\label{sect:res}

We make general assumptions regarding the distributional convergence of $X$ in two situations: (a) when centering is needed and (b) when no centering is needed. We then prove the joint distributional convergence of $(X_1,\ldots, X_k)$ for any $k\geq 2$.

We let $D$ denote the Skorokhod space of c\`{a}dl\`{a}g functions defined on $[0,\infty)$. The space $D$ is endowed with the $J_1$-topology. A comprehensive information about the $J_1$-topology can be found in \cite{Billingsley:1999} and \cite{Ethier+Kurtz:2005}. We shall write $\Rightarrow$ for weak convergence in a function space. Recall that a positive function $f$ is called {\it regularly varying} at $\infty$ of index $\alpha\in\mr$, if $\lim_{t\to\infty}(f(\lambda t)/f(t))=\lambda^\alpha$ for each $\lambda>0$. A function is called {\it slowly varying} at $\infty$ if it is regularly varying of index $0$.

Here is our assumption in the situation where a centering is needed: the limit relation holds
\begin{equation}\label{eq:1}
\Big(\frac{X(ut)-b(ut)}{a(t)}\Big)_{u\geq 0}~\Rightarrow~ (W(u))_{u\geq 0},\quad t\to\infty
\end{equation}
on $D$, where

\noindent - $a$ is a positive function satisfying $\lim_{t\to\infty}a(t)=\infty$;

\noindent - $b$ is a positive nondecreasing right-continuous function which varies regularly at $\infty$ of index $\alpha>0$ and satisfies $\lim_{t\to\infty}(b(t)/a(t))=\infty$;

\noindent - $(W(u))_{u\geq 0}$ is an a.s.\ locally H\"{o}lder continuous stochastic process, that is, $$|W(x)-W(y)|\leq V_T |x-y|^\gamma$$ for all $T>0$, $x,y\in [0,T]$, an a.s.\ finite positive random variable $V_T$ and some $\gamma\in (0,1]$. Observe that necessarily $W(0)=0$ a.s.

Put $W_1:=W$ and, for $j\geq 2$, $$W_j(u):=j\int_0^u W(u-y){\rm d}y^{\alpha (j-1)},\quad u\geq 0.$$ It can be checked that $W_j$ is a.s.\ locally H\"{o}lder continuous with exponent $\min(1, \gamma+\alpha(j-1))$.
\begin{thm}\label{thm:main1}
Suppose \eqref{eq:1}. Then
\begin{equation}\label{eq:3}
\Big(\Big(\frac{X_j(ut)-b_j(ut)}{a(t)b_{j-1}(t)}\Big)_{u\geq 0}\Big)_{j\geq 1}~\Rightarrow~ (W_j(u))_{u\geq 0})_{j\geq 1},\quad t\to\infty
\end{equation}
in the product $J_1$-topology on $D^\mn$, where $b_0(t)=1$, $b_1(t)=b(t)$ for $t>0$ and, for $j\geq 2$, $b_j=b^{\ast (j)}$, the $j$-fold Lebesgue-Stieltjes convolution of $b$ with itself.
\end{thm}

Now we state the assumption in the case where no centering is needed: there exists a positive diverging function $c$ and an a.s.\ continuous stochastic process $(Z(u))_{u\geq 0}$ such that
\begin{equation}\label{eq:100}
\Big(\frac{X(ut)}{c(t)}\Big)_{u\geq 0}~\Rightarrow~ (Z(u))_{u\geq 0},\quad t\to\infty
\end{equation}
on $D$.
\begin{thm}\label{thm:main2}
Suppose \eqref{eq:100}. Then
\begin{equation*} 
\Big(\Big(\frac{X_j(ut)}{(c(t))^j}\Big)_{u\geq 0}\Big)_{j\geq 1}~\Rightarrow~ (Z^{\ast(j)}(u))_{u\geq 0})_{j\geq 1},\quad t\to\infty
\end{equation*}
in the product $J_1$-topology on $D^\mn$.
\end{thm}

Recall that the Laplace-Stieltjes transform of a Lebesgue-Stieltjes convolution is equal to the product of the Laplace-Stieltjes transforms of the convolved functions. Therefore, it is tempting to establish counterparts of Theorems \ref{thm:main1} and \ref{thm:main2} in terms of Laplace-Stieltjes transforms and then return to the original theorems. However, as we explain at the end of Section \ref{sect:disc}, this approach may fail to work.

The a.s.\ asymptotic behavior of $X_j$ can be easily deduced when $j$ is fixed. To justify the claim, we provide just one result. As usual, $x(t)\sim y(t)$ as $t\to\infty$ means that $\lim_{t\to\infty}(x(t)/y(t))=1$.
\begin{assertion}\label{prop:almsure}
Assume that $X(t)\sim U\eee^{\beta t} t^\alpha \ell(t)$ a.s.\ as $t\to\infty$, where $\alpha>0$, $\beta\geq 0$, $U$ is a positive random variable and $\ell$ is a function slowly varying at $\infty$. If $\beta>0$, then, for $j\geq 2$, as $t\to\infty$, $$X_j(t)~\sim~\frac{(\Gamma(1+\alpha))^j}{\Gamma((1+\alpha)j)}\beta^{j-1}U^j \eee^{\beta t}t^{(1+\alpha)j-1}(\ell(t))^j\quad \text{{\rm a.s.}},$$ where $\Gamma$ is the Euler gamma-function. If $\beta=0$, then, for $j\geq 2$, as $t\to\infty$, $$X_j(t)~\sim~\frac{(\Gamma(1+\alpha))^j}{\Gamma(1+\alpha j)}U^j t^{\alpha j}(\ell(t))^j\quad \text{{\rm a.s.}}$$
\end{assertion}

\section{Discussion}\label{sect:disc}

We assume until further notice that $(T_n)_{n\geq 1}$ is a globally perturbed random walk as defined in \eqref{eq:pert} and that $X(t)=Y_1(t)=\sum_{n\geq 1}\1_{\{T_n\leq t\}}$ for $t\geq 0$.

Let $B:=(B(u))_{u\geq 0}$ be a standard Brownian motion and, for $q\geq 0$, put $B_q(u):=\int_0^u (u-y)^q{\rm d}B(y)$ for $u\geq 0$. The process $B_q:=(B_q(u))_{s\geq 0}$ is a centered Gaussian process, known in the literature as the {\it fractionally integrated Brownian motion} or the {\it Riemann–Liouville process}. It is clear that $B=B_0$. Also, integrating by parts one can check that $B_1(u)=\int_0^u B(y){\rm d}y$ and, for integer $q\geq 2$, $B_q(u)=q!\int_0^u \int_0^{u_2}\ldots\int_0^{u_q}B(y){\rm d}y{\rm d}u_q\ldots{\rm d}u_2$ for $u\geq 0$.

The following result can be found in Theorem 6 of \cite{Iksanov+Rashytov+Samoilenko:2023} under the additional assumption that $\me [\eta^a]<\infty$ for some $a>0$. It follows from Theorem 2.14 in \cite{Iksanov+Kondratenko:2025} that the additional assumption can be dispensed with.
\begin{assertion}\label{prop:sam}
Assume that $\sigma^2:={\rm Var} [\xi]\in (0,\infty)$. Then 	
$$\Big(\Big((j-1)!\frac{Y_j(ut)-\me [Y_j(ut)]}{(\mu^{-2j-1}\sigma^2 t^{2j-1})^{1/2}}\Big)_{u\geq 0}\Big)_{j\geq 1}~\Rightarrow~ ((B_{j-1}(u))_{u\geq 0})_{j\geq 1},\quad t\to\infty$$ in the product $J_1$-topology on $D^\mn$, where $\mu:=\me [\xi]<\infty$.
\end{assertion}
Now, under the same assumptions, we give a specialization of Theorem \ref{thm:main1}. First, we derive relation \eqref{eq:1} from the convergence of the first coordinate in Proposition \ref{prop:sam}. It holds with $a(t)=(\mu^{-3}\sigma^2t)^{1/2}$ and
\begin{multline*}
b(t)=\me [X(t)]=\int_{[0,\,t]}\mmp\{\eta\leq t-y\}{\rm d}\Big(1+\sum_{n\geq 1}\mmp\{\xi_1+\ldots+\xi_n\leq y\}\Big)\\~\sim~ \sum_{n\geq 1}\mmp\{\xi_1+\ldots+\xi_n\leq t\}~\sim~ \mu^{-1}t,\quad t\to\infty,
\end{multline*}
where the last asymptotic relation is secured by the elementary renewal theorem. Thus, $\alpha=1$. By Lemma \ref{lem:reg}, $b_j(t)\sim t^j/(\mu^j j!) $ as $t\to\infty$. Also, $b_j(t)=\me [Y_j(t)]$ for $t\geq 0$. Finally, $W=B$, a standard Brownian motion, and for $j\geq 1$, $W_j=jB_{j-1}$. Summarizing, we obtain the following result.
\begin{thm}\label{thm:new}
Assume that $\sigma^2={\rm Var} [\xi]\in (0,\infty)$. Then 	
$$\Big(\Big(\frac{(j-1)!}{j}\frac{X_j(ut)-\me [Y_j(ut)]}{(\mu^{-2j-1}\sigma^2 t^{2j-1})^{1/2}}\Big)_{u\geq 0}\Big)_{j\geq 1}~\Rightarrow~ ((B_{j-1}(u))_{u\geq 0})_{j\geq 1},\quad t\to\infty$$ in the product $J_1$-topology on $D^\mn$. 
\end{thm}

The distributional limit theorems for $X_j$ and $Y_j$ are identical up to the deterministic factor $1/j$. This fact may seem surprising when interpreted in terms of the decoupled and coupled branching random walks. Indeed, randomness in the limit in each generation is represented by the same Brownian motion, which arises already in the first generation. In particular, the intrinsic independence of the decoupled branching random walk induced by the branching property is not seen in the limit. Now we explain this phenomenon informally, restricting our attention to the second generation for simplicity. Referring back to formula \eqref{eq:branch}, we represent the difference $Y_2-b_2$ as the sum of a `martingale' term and a `shot noise' term: $$Y_2(t)-b_2(t)=\sum_{k\geq 1}(Y_1^{(k)}(t-T_k)-b(t-T_k))\1_{\{T_k\leq t\}}+\Big(\sum_{k\geq 1}b(t-T_k)\1_{\{T_k\leq t\}}-b_2(t)\Big)$$ for $t\geq 0$. The independence brought by the second-generation individuals is hidden in the first term, which asymptotically vanishes upon dividing by $a(t)b(t)$ (its variance is small in comparison to $(a(t)b(t))^2$). Thus, the distributional behavior of $Y_2-b_2$ is driven by the second term, which is a deterministic functional of $(T_n)_{n\geq 1}$.

Such a situation persists for general $(T_n)_{n\geq 1}$, for which $X(t)=\sum_{n\geq 1}\1_{\{T_n\leq t\}}$ satisfies \eqref{eq:1} and $v(t):={\rm Var}[X(t)]<\infty$ for all $t\geq 0$, whenever $$\lim_{t\to\infty}\frac{\int_{[0,\,t]}v(t-y){\rm d}m(y)}{(a(t)b(t))^2}=0,$$ where $m(t):=\me [X(t)]$.

In the setting of Theorem \ref{thm:main2}, that is, when the limit theorem involves no centering, the scaling limits for $X_j$ and $Y_j$ are drastically different. According to Theorem \ref{thm:main2}, the scaling limit of $X_j$ is $Z^{\ast(j)}$. However, the scaling limit of $Y_j$ is necessarily of a different form. To exemplify, we give a result which follows from Theorem 5 and Section 6 in \cite{Braganets+Iksanov:2023}.
\begin{assertion}
Let $(T_n)_{n\geq 1}$ be as given in \eqref{eq:pert}. Assume that $\mmp\{\xi>t\}\sim t^{-\beta}\ell(t)$ as $t\to\infty$ for some $\beta\in (0,1)$ and some $\ell$ slowly varying at $\infty$. Then
\begin{multline*}
\Big(\big((\mmp\{\xi>ut\})^j Y_j(ut)\big)_{u\geq 0}\Big)_{j\geq 1}\\~\Rightarrow~ \Big(\Big(\frac{1}{(\Gamma(1-\beta))^{j-1}\Gamma(1+\beta(j-1))}\int_{[0,\,u]}\mathcal{S}^\leftarrow_\beta(u-y){\rm d}y^{\beta(j-1)}\Big)_{u\geq 0}\Big)_{j\geq 1},\quad t\to\infty
\end{multline*}
in the product $J_1$-topology on $D^\mn$, where $\mathcal{S}^\leftarrow_\beta(u):=\inf\{v\geq 0: \mathcal{S}_\beta(v)>u\}$ for $u\geq 0$, $(\mathcal{S}_\beta(v))_{v\geq 0}$ is a drift-free $\beta$-stable subordinator with $-\log \me [\eee^{-s\mathcal{S}_\beta(v)}]=\Gamma(1-\beta)vs^\beta$ for $s\geq 0$, and $\Gamma$ is the Euler gamma-function.
\end{assertion}

We close this section by justifying the claim made in Section \ref{sect:res} regarding the use of Laplace-Stieltjes transforms. Assuming that \eqref{eq:1} holds, we would like to deduce that
\begin{equation}\label{eq:lst}
\Big(\frac{\int_{[0,\,\infty)}\eee^{-uy/t}{\rm d}(X(y)-b(y))}{a(t)}\Big)_{u>0}~\Rightarrow~\Big(u\int_0^\infty \eee^{-uy}W(y){\rm d}y\Big)_{u>0},\quad t\to\infty
\end{equation}
on $D((0,\infty))$ (the set of c\`{a}dl\`{a}g functions defined on $(0,\infty)$). However, the Laplace-Stieltjes functional is not continuous in the $J_1$-topology, see p.~258 in \cite{Whitt:1972}. Thus, the implication does not follow automatically and requires a proof. Without going into details we state that \eqref{eq:1} entails \eqref{eq:lst} provided that $$\lim_{T\to\infty}\limsup_{t\to\infty}\mmp\Big\{\max_{u\in [c,\,d]}\Big|\int_T^\infty \eee^{-uy}(X(ty)-b(ty)){\rm d}y\Big|>\varepsilon a(t)\Big\}=0$$ for all $\varepsilon>0$ and all $c,d\in  (0,\infty)$, $c<d$. This is a sufficient condition, we do not claim it is necessary. If \eqref{eq:lst} holds, then it is not that hard to derive a counterpart of \eqref{eq:3}:
\begin{equation}\label{eq:lst2}
\Big(\Big(\frac{\int_{[0,\,\infty)}\eee^{-uy/t}{\rm d}(X_j(y)-b_j(y))}{a(t)b_{j-1}(t)}\Big)_{u>0}\Big)_{j\geq 1}~\Rightarrow~\Big(\Big(u\int_0^\infty \eee^{-uy}W_j(y){\rm d}y\Big)_{u>0}\Big)_{j\geq 1},\quad t\to\infty
\end{equation}
on $D((0,\infty))^\mn$. We only address the convergence of the second coordinate. Put $\varphi(s):=\int_{[0,\,\infty)}\eee^{-sy}{\rm d}X(y)$ and $\psi(t):=\int_{[0,\,\infty)}\eee^{-sy}{\rm d}b(y)$ for $s>0$. Write, for $u>0$ and large $t$,
\begin{multline*}
\frac{\int_{[0,\,\infty)}\eee^{-uy/t}{\rm d}(X_2(y)-b_2(y))}{a(t)b(t)}=\frac{(\varphi(u/t))^2-(\psi(u/t))^2}{a(t)b(t)}=\frac{a(t)}{b(t)}\Big(\frac{\varphi(u/t)-\psi(u/t)}{a(t)}\Big)^2\\+2\frac{\psi(u/t)}{b(t)}\frac{\varphi(u/t)-\psi(u/t)}{a(t)}.
\end{multline*}
Since $\lim_{t\to\infty}(a(t)/b(t))=0$ and \eqref{eq:lst} holds, we conclude that, for all $c,d\in (0,\infty)$, $c<d$, $$\lim_{t\to\infty}\frac{a(t)}{b(t)}\sup_{u\in [c,\,d]}\Big(\frac{\varphi(u/t)-\psi(u/t)}{a(t)}\Big)^2=0\quad\text{in probability.}$$ By Karamata's Tauberian theorem (Theorem 1.7.1 in \cite{BGT}), $$\lim_{t\to\infty}\frac{\psi(u/t)}{b(t)}=\Gamma(1+\alpha)u^{-\alpha}=\int_0^\infty \eee^{-ux}{\rm d}x^\alpha.$$ Furthermore, this convergence is uniform in $u\in [c,\,d]$. Hence,
\begin{multline*}
\Big(\frac{\int_{[0,\,\infty)}\eee^{-uy/t}{\rm d}(X_2(y)-b_2(y))}{a(t)b(t)}\Big)_{u>0}~\Rightarrow~\Big(2u\int_0^\infty \eee^{-ux}{\rm d}x^\alpha \int_0^\infty \eee^{-uy}W(y){\rm d}y\Big)_{u>0}\\=\Big(u\int_0^\infty \eee^{-uy}W_2(y){\rm d}y\Big)_{u>0},\quad t\to\infty
\end{multline*}
on $D((0,\infty))$. A final issue to mention is that it is not immediately clear how to deduce \eqref{eq:3} from \eqref{eq:lst2}.

\section{Proofs}

We state right away that Proposition \ref{prop:almsure} follows from Lemma \ref{lem:reg} in the case $\beta=0$ and from Lemma \ref{lem:exp} in the case $\beta>0$.

\subsection{Proof of Theorem \ref{thm:main1}}

Since each coordinate on the left-hand side of \eqref{eq:3} is a deterministic functional of $X$, it is enough to prove the distributional convergence of each coordinate. The joint distributional convergence then follows automatically.

We use mathematical induction. Assume that we have already proved that
\begin{equation}\label{eq:assump}
\Big(\frac{X_k(ut)-b_k(ut)}{a(t)b_{k-1}(t)}\Big)_{u\geq 0}~\Rightarrow~ (W_k(u))_{u\geq 0},\quad t\to\infty
\end{equation}
on $D$. To derive the analogous asymptotic relation with $k+1$ replacing $k$ we have to show the weak convergence of finite-dimensional distributions and tightness. Let $(t_n)_{n\geq 1}$ be any sequence of positive numbers satisfying $\lim_{n\to\infty}t_n=\infty$. In view of the Cram\'{e}r-Wold device, the former follows if we can show that, for all $m\geq 1$, any $0\leq u_1<\ldots<u_m<\infty$ and any $\lambda_1,\ldots$ $\lambda_m\in\mr$,
\begin{equation}\label{eq:findim2}
\sum_{i=1}^m \lambda_i \frac{X_{k+1}(u_it_n)-b_{k+1}(u_it_n)}{a(t_n)b_k(t_n)}~\dod~ \sum_{i=1}^m \lambda_i W_{k+1}(u_i), \quad n\to\infty,
\end{equation}
where $\dod$ denotes one-dimensional convergence in distribution. Here is a basic decomposition for what follows:
\begin{multline*}
X_{k+1}(ut)-b_{k+1}(ut)\\=\int_{[0,\,ut]}(X(ut-y)-b(ut-y)){\rm d}X_k(y)+\int_{[0,\,ut]}(X_k(ut-y)-b_k(ut-y)){\rm d}b(y)\\=\int_{[0,\,u]}(X((u-y)t)-b((u-y)t)){\rm d}_y X_k(yt)+\int_{[0,\,u]}(X_k((u-y)t)-b_k((u-y)t)){\rm d}_yb(yt),
\end{multline*}
where ${\rm d}_y$ denotes integration over $y$.

Relation \eqref{eq:assump} ensures that, for all $T>0$, $\lim_{t\to\infty}\sup_{y\in [0,\,T]}|X_k(yt)/b_k(t)- b_k(yt)/b_k(t)|=0$ in probability. By Lemma \ref{lem:reg}, $b_k$ is regularly varying at $\infty$ of index $\alpha k$. As a consequence, for all $T>0$, $\lim_{t\to\infty}\sup_{y\in [0,\,T]}|X_k(yt)/b_k(t)-y^{\alpha k}|=0$ in probability. For $i\geq 1$, denote by $\nu_i$ the measure defined by $\nu_i([0,\,y])=y^{\alpha i}$ for $y\geq 0$.  By the Skorokhod representation theorem, for each $n\geq 1$, there exists $X^\ast_{k,\,t_n}$, a version of $(X_k(t_n u))_{u\geq 0}$, such that the random measures $\nu^\ast_{k,\,t_n}$ defined by $\nu^\ast_{k,\,t_n}([0,\,y])=X^\ast_{k,\,t_n}(y)$ converge weakly to the continuous measure $\nu_k$ a.s. Also, on the same probability space there exist versions $\hat X_{1,\,t_n}$ and $\hat X_{k,\,t_n}$ of the first and the $k$th coordinate on the left-hand side of \eqref{eq:3}, with $t_n$ replacing $t$, which converge a.s. to $\hat W_1$ and $\hat W_k$, versions of $W$ and $W_k$, in the $J_1$-topology on $D$, as $n\to\infty$. Furthermore, for all $T>0$,
\begin{equation}\label{eq:cont}
\lim_{n\to\infty}\sup_{u\in [0,\,T]}|\hat X_{1,\,t_n}(u)-\hat W_1(u)|=0\quad \text{and}\quad \lim_{n\to\infty}\sup_{u\in [0,\,T]}|\hat X_{k,\,t_n}(u)-\hat W_k(u)|=0\quad\text{a.s.}
\end{equation}
because $W$ and $W_k$ are a.s.\ continuous, and so are $\hat W_1$ and $\hat W_k$. By assumption, $b$ varies regularly at $\infty$ of index $\alpha$. Hence, the measures $\nu_{1,\,t_n}$ 
defined by $\nu_{1,\,t_n}([0,\,y])=b(yt_n)/b(t_n)$ for $y>0$ converge weakly to the continuous measure $\nu_1$. Thus, by Lemma \ref{lem:conv},
\begin{multline*}
\lim_{n\to\infty}\sum_{i=1}^m\lambda_i \Big(\int_{[0,\,u_i]}\hat X_{1,\,t_n}(u_i-y){\rm d}\Big(\frac{X^\ast_{k,\,t_n}(y)}{b_k(t_n)}\Big)+\int_{[0,\,u_i]}\hat X_{k,\,t_n}(u_i-y){\rm d}_y\Big(\frac{b(yt_n)}{b(t_n)}\Big)\Big)\\=\sum_{i=1}^m\lambda_i\Big( \int_{[0,\,u_i]}\hat W_1(u_i-y){\rm d}y^{\alpha k}+ \int_{[0,\,u_i]}\hat W_k(u_i-y){\rm d}y^\alpha\Big) \quad\text{a.s.}
\end{multline*}
Now we conclude that
\begin{multline*}
\sum_{i=1}^m \lambda_i \frac{X_{k+1}(u_it_n)-b_{k+1}(u_it_n)}{a(t_n)b_k(t_n)}~\dod~\sum_{i=1}^m \lambda_i\Big(\int_{[0,\,u_i]} W(u_i-y){\rm d}y^{\alpha k}\\+\frac{\Gamma(1+\alpha k)}{\Gamma(1+\alpha)\Gamma(1+\alpha(k-1))} \int_{[0,\,u_i]} W_k(u_i-y){\rm d}y^\alpha\Big),\quad n\to\infty
\end{multline*}
having utilized $$\lim_{t\to\infty}\frac{b_{k-1}(t)b(t)}{b_k(t)}=\frac{\Gamma(1+\alpha k)}{\Gamma(1+\alpha)\Gamma(1+\alpha(k-1))},$$ see Lemma \ref{lem:reg}. Noting that $$\int_0^y (y-x)^{\alpha(k-1)}{\rm d}x^\alpha=\frac{\Gamma(1+\alpha)\Gamma(1+\alpha(k-1))}{\Gamma(1+\alpha k)}y^{\alpha k}$$ we infer
\begin{multline*}
\int_{[0,\,u]} W(u-y){\rm d}y^{\alpha k}+\frac{\Gamma(1+\alpha k)}{\Gamma(1+\alpha)\Gamma(1+\alpha(k-1))} \int_{[0,\,u]} W_k(u-y){\rm d}y^\alpha=\int_{[0,\,u]} W(u-y){\rm d}y^{\alpha k}\\+\frac{\Gamma(1+\alpha k)}{\Gamma(1+\alpha)\Gamma(1+\alpha(k-1))}k \int_{[0,\,u]} W(u-y){\rm d}\Big(\int_0^y (y-x)^{\alpha(k-1)}{\rm d}x^\alpha\Big)=\int_{[0,\,u]} W(u-y){\rm d}y^{\alpha k}\\+k\int_{[0,\,u]} W(u-y){\rm d}y^{\alpha k}=(k+1)\int_{[0,\,u]} W(u-y){\rm d}y^{\alpha k}=W_{k+1}(u).
\end{multline*}
thereby completing the proof of \eqref{eq:findim2}.

Now we pass to discussing tightness. Fix any $T>0$. The processes $$\Big(\Big(\frac{X_{k+1}(ut_n)-b_{k+1}(ut_n)}{a(t_n)b_k(t_n)}\Big)_{u\geq 0}\Big)_{n\geq 1}$$ have the same distributions as $$\Big(\Big(\int_{[0,\,u]}\hat X_{1,\,t_n}(u-y){\rm d}\Big(\frac{X^\ast_{k,\,t_n}(y)}{b_k(t_n)}\Big)+\frac{b_{k-1}(t_n)b(t_n)}{b_k(t_n)}\int_{[0,\,u]}\hat X_{k,\,t_n}(u-y){\rm d}_y\Big(\frac{b(yt_n)}{b(t_n)}\Big)\Big)_{u\geq 0}\Big)_{n\geq 1}.$$ Further, 
\begin{multline*}
\sup_{u\in [0,\,T]}\Big|\int_{[0,\,u]}(\hat X_{1,\,t_n}(u-y)-\hat W_1(u-y)){\rm d}\Big(\frac{X^\ast_{k,\,t_n}(y)}{b_k(t_n)}\Big)\Big|\\\leq \sup_{u\in [0,\,T]}|\hat X_{1,\,t_n}(u)-\hat W_1(u)|\frac{X^\ast_{k,\,t_n}(T)}{b_k(t_n)}~\to~0,\quad n\to\infty\quad\text{a.s.}
\end{multline*}
in view of \eqref{eq:cont} and $\lim_{n\to\infty} X^\ast_{k,\,t_n}(T)/b_k(t_n)=T^{\alpha k}$ a.s. Analogously, $$\lim_{n\to\infty}\sup_{u\in [0,\,T]}\Big|\int_{[0,\,u]}(\hat X_{k,\,t_n}(u-y)-\hat W_k(u-y)){\rm d}\Big(\frac{b(yt_n)}{b(t_n)}\Big)\Big|=0 \quad\text{a.s.}$$ Therefore, it is enough to prove tightness in $D([0,T])$ of the distributions of $$\Big(\Big(\int_{[0,\,u]}\hat W_1(u-y){\rm d}\Big(\frac{X^\ast_{k,\,t_n}(y)}{b_k(t_n)}\Big)\Big)_{u\geq 0}\Big)_{n\geq 1}\quad\text{and}\quad \Big(\Big(\int_{[0,\,u]}\hat W_k(u-y){\rm d}_y\Big(\frac{b(yt_n)}{b(t_n)}\Big)\Big)_{u\geq 0}\Big)_{n\geq 1}.$$ By Corollary on p.~142 in \cite{Billingsley:1999}, it suffices to show that, for any positive $\theta_1$ and $\theta_2$, there exist $\delta>0$ and integer $n_0\geq 1$ such that $$\mmp\Big\{\sup_{0\leq u,v\leq T,\,|u-v|\leq \delta}\Big|\int_{[0,\,u]}\hat W_1(u-y){\rm d}X^\ast_{k,\,t_n}(y)-\int_{[0,\,v]}\hat W_1(v-y){\rm d}X^\ast_{k,\,t_n}(y)\Big|>\theta_1 b_k(t_n)\Big\}\leq \theta_2$$ for all $n\geq n_0$, and that, for any positive $\theta_3$ and $\theta_4$, there exist $\rho>0$ and integer $n_1>0$ such that
\begin{equation}\label{eq:tight}
\mmp\Big\{\sup_{0\leq u,v\leq T,\,|u-v|\leq\rho}\Big|\int_{[0,\,u]}\hat W_k(u-y){\rm d}_yb(yt_n)-\int_{[0,\,v]}\hat W_k(v-y){\rm d}_yb(yt_n)\Big|>\theta_3 b(t_n)\Big\}\leq \theta_4
\end{equation}
for all $n\geq n_1$.

We only prove \eqref{eq:tight}, the proof of the inequality involving $\hat W_1$ being similar. The process $\hat W_k$ is a.s.\ locally H\"{o}lder continuous with exponent $\lambda:=\min(1, \gamma+\alpha(k-1))$, that is, $|\hat W_k(x)-\hat W_k(y)|\leq \hat V_{k,\,T}|x - y|^\lambda$ for $x,y\in [0,T]$ and some a.s.\ finite positive random variable $\hat V_{k,\,T}$. Putting $\hat W_k(t)=0$ for $t<0$ we show that the last inequality actually holds for $x,y\in (-\infty, T]$. A proof is only needed for the case $\min(x,y)\leq 0<\max(x,y)$. Here it is: $$|\hat W_k(x)-\hat W_k(y)|=|\hat W_k(\max(x,y))|\leq \hat V_{k,\,T}(\max(x,y))^\lambda\leq \hat V_{k,\,T}|x-y|^\lambda.$$
Let $0\leq v<u\leq T$ and $u-v\leq \rho$ for some $\rho>0$. Then
\begin{multline*}
\Big|\int_{[0,\,u]}\hat W_k(u-y){\rm d}_y\Big(\frac{b(yt_n)}{b(t_n)}\Big)-\int_{[0,\,v]}\hat W_k(v-y){\rm d}_y\Big(\frac{b(yt_n)}{b(t_n)}\Big)\Big|\\\leq \int_{[0,\,u]}|\hat W_k(u-y)-\hat W_k(v-y)|{\rm d}_y\Big(\frac{b(yt_n)}{b(t_n)}\Big)\leq \hat V_{k,\,T}(u-v)^\lambda\frac{b(Tt_n)}{b(t_n)}\leq 2\hat V_{k,\,T}\rho^\lambda T^\alpha
\end{multline*}
for $n\geq n_1$ and $n_1$ large enough. This proves \eqref{eq:tight} in the case $0\leq v<u\leq T$. The complementary case $0\leq u<v\leq T$ can be investigated analogously.

The proof of Theorem \ref{thm:main1} is complete.

\subsection{Proof of Theorem \ref{thm:main2}}

This proof is much simpler than the previous one. We only prove the distributional convergence of each coordinate.

We use mathematical induction. Assume that we have already proved that $$\Big(\frac{X_k(ut)}{(c(t))^k}\Big)_{u\geq 0}~\Rightarrow~ (Z^{\ast(k)}(u))_{u\geq 0},\quad t\to\infty$$ on $D$. Since $Z$ is a.s.\ continuous, so is $Z^{\ast (k+1)}$. The convergence of nondecreasing functions to a continuous limit is locally uniform. Therefore, it is enough to prove the weak convergence of finite-dimensional distributions. According to the Cram\'{e}r-Wold device, the latter is equivalent to the distributional convergence
\begin{equation}\label{eq:findim}
\sum_{i=1}^m \lambda_i \frac{X_{k+1}(u_it)}{(c(t))^{k+1}}=\sum_{i=1}^m \lambda_i \int_{[0,\,u_i]}\frac{X_k((u_i-y)t)}{(c(t))^k}{\rm d}_y\Big(\frac{X(yt)}{c(t)}\Big)~\dod~\sum_{i=1}^m \lambda_i Z^{\ast(k+1)}(u_i),\quad t\to\infty
\end{equation}
holding for all $m\geq 1$, any $0\leq u_1<\ldots<u_m<\infty$ and any $\lambda_1,\ldots$ $\lambda_m\geq 0$.

Now we prove \eqref{eq:findim}. Let $(t_n)_{n\geq 1}$ be any sequence of positive numbers satisfying $\lim_{n\to\infty}t_n=\infty$. By the Skorokhod representation theorem, for each $n\geq 1$, there exists $\hat X_{k,\,t_n}$, a version of $(X_k(ut_n)/(c(t_n))^k)_{u\geq 0}$, and $\hat Z_k$, a version of $Z^{\ast(k)}$, such that $\lim_{n\to\infty}\hat X_{k,\,t_n}=\hat Z_k$ a.s.\ on $D$. Also, on the same probability space, for each $n\geq 1$, there exists $\hat X_{1,\,t_n}$, a version of $(X(ut_n)/c(t_n))_{u\geq 0}$ and $\hat Z_1$, a version of $Z$, such that the random measures $\nu_{1,\,t_n}$ defined by $\nu_{1,\,t_n}([0,\,y])=\hat X_{1,\,t_n}(y)$ for $y>0$ converge weakly a.s.\ to the continuous measures $\nu_1$ given by $\nu_1([0,\,y])=\hat Z_1(y)$ for $y>0$. Thus, by Lemma \ref{lem:reg},
\begin{equation*}
\lim_{n\to\infty} \sum_{i=1}^m \lambda_i \int_{[0,\,u_i]}\hat X_{k,\,t_n}(u_i-y){\rm d}_y \hat X_{1,\,t_n}(y)=\sum_{i=1}^m \int_{[0,\,u_i]}\hat Z_k(u_i-y){\rm d}\hat Z_1(y) \quad\text{a.s.}
\end{equation*}
Since the diverging sequence $(t_n)_{n\geq 1}$ is arbitrary, this entails
\begin{multline*}
\sum_{i=1}^m \lambda_i \int_{[0,\,u_i]}\frac{X_k((u_i-y)t)}{(c(t))^k}{\rm d}_y\Big(\frac{X(yt)}{c(t)}\Big)~\dod~\sum_{i=1}^m \lambda_i\int_{[0,\,u_i]} Z^{\ast(k)}(u_i-y){\rm d}Z(y)\\=\sum_{i=1}^m \lambda_i Z^{\ast(k+1)}(u_i),\quad t\to\infty.
\end{multline*}
Thus, \eqref{eq:findim} does indeed hold.

The proof of Theorem \ref{thm:main2} is complete.

\section{Appendix}

Here we collect a few auxiliary facts. The first of these can be found, for instance, in Lemma A.5 of \cite{Iksanov:2013}.
\begin{lemma}\label{lem:conv}
Let $0\leq c<d<\infty$. Assume that $\lim_{t\to\infty}x_t = x$ on $D$ and that, as $t\to\infty$, finite measures $\nu_t$ converge weakly on $[c, d]$ to a finite continuous measure $\nu$. Then $$\lim_{t\to\infty}\int_{[c,\,d]}
x_t(y){\rm d}\nu_t(y)=\int_{[c,\,d]}x(y){\rm d}\nu(y).$$
\end{lemma}

\begin{lemma}\label{lem:reg}
Let $f$ be a nondecreasing right-continuous function which varies regularly at $\infty$ of index $\alpha>0$. Then, for $j\geq 2$,
\begin{equation}\label{eq:2}
\lim_{t\to\infty}\frac{f^{\ast(j)}(t)}{(f(t))^j}=\frac{(\Gamma(1+\alpha))^j}{\Gamma(1+\alpha j)},
\end{equation}
where $\Gamma$ is the Euler gamma-function. In particular, $f^{\ast (j)}$ is regularly varying at $\infty$ of index $\alpha j$.
\end{lemma}
\begin{proof}
This can be checked either directly or with the aid of Laplace-Stieltjes transforms and Tauberian theorems.

We give an even simpler proof, which is based on an application of Lemma \ref{lem:conv} and mathematical induction. Assume that relation \eqref{eq:2} holds with $j=k$. Since the convergence of monotone functions to a continuous limit is locally uniform, we conclude that, for all $T>0$, $\lim_{t\to\infty}\sup_{y\in [0,\,T]}|f^{\ast(k)}(yt)/f^{\ast(k)}(t)-y^{\alpha k}|=0$. For each $t>0$, the measure $\nu_t$ defined by $\nu_t([0,\,y])=f(yt)/f(t)$ for $y>0$ converges weakly to the finite continuous measure $\nu$ given by $\nu([0,\,y])=y^\alpha$. Hence, by Lemma \ref{lem:conv}, $$\frac{f^{\ast(k+1)}(t)}{f^{\ast(k)}(t)f(t)}=\frac{\int_{[0,\,1]}f^{\ast(k)}((1-y)t){\rm d}_y f(yt)}{f^{\ast(k)}(t)f(t)}=\alpha\int_0^1 (1-y)^{\alpha k}y^{\alpha-1}{\rm d}y=\frac{\Gamma(1+\alpha)\Gamma(1+\alpha k)}{\Gamma(1+\alpha(k+1))}.$$ An application of \eqref{eq:2} with $j=k$ demonstrates that relation \eqref{eq:2} also holds with $j=k+1$. This completes the proof.
\end{proof}

\begin{lemma}\label{lem:exp}
Let $f$ be a nondecreasing right-continuous function which satisfies $$f(t)~\sim~A\eee^{\beta t}t^\alpha\ell(t),\quad t\to\infty,$$ where $A, \alpha$ and $\beta$ are positive constants and $\ell$ is a function slowly varying at $\infty$. Then, for $j\geq 2$,
\begin{equation}\label{eq:23}
f^{\ast(j)}(t)~\sim~\frac{(\Gamma(1+\alpha))^j}{\Gamma((1+\alpha)j)}\beta^{j-1} A^j \eee^{\beta t} t^{(1+\alpha)j-1}(\ell(t))^j,\quad t\to\infty,
\end{equation}
where $\Gamma$ is the Euler gamma-function.
\end{lemma}
\begin{proof}
We use mathematical induction. Assume that relation \eqref{eq:23} holds true with $j=k$. Put $f_{\beta,\,k}(t):=\eee^{-\beta t}f^{\ast(k)}(t)$ and $\hat f_\beta(t):=\int_{[0,\,t]}\eee^{-\beta y}{\rm d}f(y)$ for $t\geq 0$. The induction assumption is equivalent to $$f_{\beta,\,k}(t)~\sim~\frac{(\Gamma(1+\alpha))^k}{\Gamma((1+\alpha)k)}\beta^{k-1} A^k t^{(1+\alpha)k-1}(\ell(t))^k,\quad t\to\infty.$$ Integrating by parts and using Karamata's theorem (Proposition 1.5.8 in \cite{BGT}) we conclude that $$\hat f_\beta(t)\sim \frac{\beta A}{1+\alpha}t^{1+\alpha}\ell(t),\quad t\to\infty.$$ An application of Lemma \ref{lem:conv} (justification of its applicability is analogous to that in the proof of Lemma \ref{lem:reg}) yields
\begin{multline*}
\eee^{-\beta t}\frac{f^{\ast(k+1)}(t)}{f_{\beta,\,k}(t)\hat f_\beta(t)}=\eee^{-\beta t}\frac{\int_{[0,\,1]}f_{\beta,\,k}((1-y)t){\rm d}_y \hat f_\beta(yt)}{f_{\beta,\,k}(t)\hat f_\beta(t)}~\to~(1+\alpha)\int_0^1 (1-y)^{(1+\alpha)k-1} y^\alpha{\rm d}y\\=(1+\alpha)\frac{\Gamma((1+\alpha)k)\Gamma(1+\alpha)}{\Gamma((1+\alpha)(k+1))},\quad t\to\infty
\end{multline*}
thereby showing that relation \eqref{eq:23} holds with $j=k+1$.
\end{proof}

\vskip0.5cm
\noindent
{\bf Acknowledgement}. The work of Wissem Jedidi was supported by
the Ongoing Research Funding program, (ORF-2025-162), King Saud University, Riyadh, Saudi Arabia.

\end{document}